\documentclass[11pt,reqno]{amsart}
\usepackage[letterpaper,margin=1in]{geometry}                
\usepackage{amssymb,amsmath,amsfonts,stmaryrd}
\usepackage{xcolor}
\usepackage[colorlinks,
             linkcolor=blue,
             citecolor=green!70!black,
             pdfproducer={pdfLaTeX},
             pdfpagemode=None,
             bookmarksopen=true
             bookmarksnumbered=true]{hyperref}
\usepackage{graphicx}
\usepackage{multicol}

\usepackage{tikz}
\usetikzlibrary{decorations.markings,arrows,decorations.pathreplacing}
\usepackage{arydshln}

\newcommand\arXiv[1]{\href{http://arxiv.org/abs/#1}{\nolinkurl{arXiv:#1}}}
\newcommand\MRnumber[1]{\href{http://www.ams.org/mathscinet-getitem?mr=#1}{\nolinkurl{MR#1}}}
\newcommand\DOI[1]{\href{http://dx.doi.org/#1}{\nolinkurl{DOI:#1}}}
\newcommand\MAILTO[1]{\href{mailto:#1}{\nolinkurl{#1}}}

\newtheorem*{mainthm}{Main Theorem}

\newtheorem{dummy}{Dummy}[section]

\newtheorem{proposition}[dummy]{Proposition}

\newtheorem{theorem}[dummy]{Theorem}
\newtheorem{definition}[dummy]{Definition}

\theoremstyle{definition}

\newtheorem{remark}[dummy]{Remark}

\usepackage{dsfont}
\renewcommand\mathbb\mathds

\newcommand\bC{\mathbb C}

\newcommand\bQ{\mathbb Q}

\newcommand\bZ{\mathbb Z}

\newcommand\cB{\mathcal B}
\newcommand\cC{\mathcal C}
\newcommand\cD{\mathcal D}
\newcommand\cE{\mathcal E}
\newcommand\cF{\mathcal F}

\newcommand\cI{\mathcal I}

\newcommand\cM{\mathcal M}
\newcommand\cN{\mathcal N}

\newcommand\rB{\mathrm B}

\usepackage[mathscr]{euscript}

\newcommand\longto\longrightarrow
\newcommand\mono\hookrightarrow
\newcommand\epi\twoheadrightarrow

\newcommand\<\langle
\renewcommand\>\rangle
\newcommand\sminus\smallsetminus

\newcommand\st{\text{ s.t.\ }}

\newcommand\id{\mathrm{id}}

\DeclareMathOperator\Mod{\cat{Mod}}
\DeclareMathOperator\Rep{\cat{Rep}}

\renewcommand\ell{\mathrm{ell}}

\newcommand\br{\mathrm{br}}

\DeclareMathOperator\homology{H}
\renewcommand\H{\homology}

\DeclareMathOperator\FPdim{FPdim}
\DeclareMathOperator\K{K}

\DeclareMathOperator\Gal{Gal}

\DeclareMathOperator\NBFC{NBFC}
\DeclareMathOperator\Witt{Witt}
\newcommand\triv{\mathrm{triv}}
\DeclareMathOperator\MNE{MNE}

\renewcommand\Vec{\cat{Vec}}
\newcommand\sVec{\cat{sVec}}

\renewcommand\d{\mathrm{d}}

\newcommand\define[1]{\emph{#1}}
\newcommand\cat[1]{\mathbf{#1}}

\title{Schopieray's Galois-modular extension conjecture}
\author{Theo Johnson-Freyd}
\date{\today}
\thanks{I thank Andrew Schopieray for suggesting and discussing this question and for helpful comments on a draft. I thank both anonymous referees for their comments and improvements, and in particular one of the referees for suggesting how to remove my pseudounitarity assumption from Proposition~\ref{prop:trivialaction}. This article was written at the BIRS workshop 26w5586 - New Developments in Tensor Categories. The author's work is supported by the NSERC grant RGPIN-2021-02424 and by the Simons Collaboration on Global Categorical Symmetries SFI-MPS-GCS-00008528-16.  AI technologies were not used in this research.}

\begin{document}
\begin{abstract}
Plavnik, Schopieray, Yu, and Zhang have drawn attention to those (automatically premodular) fusion subcategories of modular fusion categories which are submodules for the Galois action on the ambient category. In particular, they showed that a subcategory is a Galois submodule if and only if its centralizer is integral.  In the other direction, Schopieray has conjectured that every premodular fusion category can be embedded as a Galois-closed subcategory of a modular category; Schopieray calls such an embedding a \define{Galois-modular extension}. We prove Schopieray's conjecture for pseudounitary categories; in the nonpseudounitary case, we prove that every braided fusion category admits a ``Galois nondegenerate extension'' but we do not know it to be modular. Along the way we record some general comments about the minimal nondegenerate extension problem for braided fusion categories.
\end{abstract}
\maketitle

\section{Galois-modular extensions}

Braided fusion categories are a particularly nice, and very well-studied, type of 
categorified algebra, behaving in many ways like higher analogues of finite-dimensional semisimple algebras.
 All of our braided fusion categories will be over~$\bC$. We will not review the definition;  standard references include \cite{MR2609644,EGNO}.

 Like their lower-categorical cousins, one of the most basic invariants of a braided fusion category is its centre. Specifically, the appropriate \define{centre} of a braided fusion category $\cB$ is its ``M\"uger'' or ``symmetric'' centre (and in this paper, ``centre'' without further adjective always means the M\"uger centre):
$$ Z_2(\cB) := \{x \in \cB \st \br_{y,x} \circ \br_{x,y} = \id_{x \otimes y} \text{ for all } y \in \cB \} \subset \cB.$$
Here $\br_{x,y} : x \otimes y \to y \otimes x$ denotes the braiding in $\cB$, and implicit in the formula is to take the full category with the declared objects. As in the lower-categorical case, this centre can vary from trivial to everything. At one extreme, a braided fusion category is \define{symmetric} when $Z_2(\cB) = \cB$. At the other extreme, a braided fusion category is \define{nondegenerate} when $Z_2(\cB) = \cat{Vec}$. Most braided fusion categories are somewhere between these two extremes.

Every braided fusion category can be embedded fully faithfully inside a nondegenerate braided fusion category, just as any finite-dimensional associative algebra can be embedded faithfully inside a central simple algebra. In the lower-categorical case, one can embed $A$ inside the algebra of matrices on the underlying vector space of $A$. In the higher-categorical case, one can embed $\cB$ inside the Drinfel'd centre $Z_1(\cB)$ of the underlying monoidal category of $\cB$.
More generally, a \define{nondegenerate extension} of $\cB$ is a choice of fully faithful braided embedding $\cB \mono \cM$ where $\cM$ is a nondegenerate braided fusion category.

A good way to measure the ``relative size''  of a nondegenerate extension $\cB \mono \cM$ is to compute its centralizer:
$$ Z_2(\cB \mono \cM) := \{x \in \cM \st \br_{y,x} \circ \br_{x,y} = \id_{x \otimes y} \text{ for all } y \in \cB \} \subset \cM.$$
It is easy to see that the M\"uger centres of $\cB$ and $Z_2(\cB \mono \cM)$ are (canonically) equivalent, so that in particular $Z_2(\cB \mono \cM)$ is never smaller than $Z_2(\cB)$, and M\"uger suggested to focus on those extensions that saturate this bound \cite{MR1990929}:
\begin{definition}
  A nondegenerate extension $\cB \mono \cM$ of a braided fusion category is \define{minimal} when its centralizer $Z_2(\cB \mono \cM)$ is symmetric.
\end{definition}
\begin{remark}\label{remark:mne}
  In spite of the name, there are known examples of braided fusion categories which do not admit any nondegenerate extension which is minimal in M\"uger's sense. In other words, M\"uger's bound is not always sharp. It is known that a minimal nondegenerate extension exists whenever $Z_2(\cB) \cong \sVec$ \cite{MNE}. When $Z_2(\cB) \cong \cat{Rep}(G)$ for a finite group $G$, then there is a class $O_4(\cB) \in \H^4(G; \bC^\times)$, and $\cB$ admits a minimal nondegenerate extension if and only if $O_4(\cB) = 1$ \cite{MR4971769}; for every class $\alpha \in \H^4(G; \bC^\times)$, there exists a $\cB$ with $O_4(\cB) = \alpha$~\cite{2601.09060}.
\end{remark}

The vast majority of the literature on nondegenerate extensions has focused on the question of constructing and obstructing these so-called minimal extensions. But symmetric categories are not the only interesting class of braided fusion category. Recall that each object $x$ in a (braided) fusion category $\cB$ has a \define{Frobenius--Perron dimension} $\FPdim(x)$, defined as the Frobenius--Perron eigenvalue of the action of multiplication by $[x]$ on the Grothendieck ring $\K_0(\cB)$ of $\cB$. The following definition is from \cite{MR2183279}:

\begin{definition}
  A (braided) fusion category is \define{integral} if the Frobenius--Perron dimensions of all of its objects are rational integers.
\end{definition}

Symmetric implies integral, and the integral (braided) fusion categories are under relatively good control due to their relation to (quasitriangular) quasiHopf algebras. This paper is interested in studying nondegenerate extensions with integral centralizer. The following name is suggested by Schopieray \cite{SchopierayBIRS}:
\begin{definition} \label{defn:Galoisextension}
  A nondegenerate extension $\cB \mono \cM$ is \define{Galois} when $Z_2(\cB \mono \cM)$ is integral.
\end{definition}
This name deserves some explanation. First, recall that a \define{modular} fusion category is a nondegenerate braided fusion category equipped with a spherical structure. A \define{premodular} fusion category is a spherical braided fusion category which might  be degenerate. A fusion subcategory of a modular category is automatically premodular. Conversely, a \define{modular extension} of a premodular fusion category is an extension $\cB \mono \cM$ with $\cM$ modular.

Second, recall that if $\cM$ is modular, then its set of simple objects comes with a canonical action by $\Gal(\bar{\bQ}/\bQ)$, which is trivial only when $\cM$ is integral. The relation to Galois extensions is identified by Plavnik, Schopieray, Yu, and Zhang:
\begin{theorem}[Theorem 4.1.6 of \cite{MR4834521}]
  A modular extension $\cB \mono \cM$ is Galois in the sense of Definition~\ref{defn:Galoisextension} if and only if (the set of isomorphism classes of simple objects of) $\cB$ is closed under the Galois action on (the set of isomorphism classes of simple objects of) $\cM$.
\end{theorem}
Galois-modular extensions provide computational control: 
if $\cB \mono \cM$ is Galois-modular, then the $\Gal(\bar{\bQ}/\bQ)$-action on $\cM$ restricts to one on $\cB$, and this restricted action often supplies a lot of information about $\cB$, just as the $\Gal(\bar{\bQ}/\bQ)$-actions on modular fusion categories have proven useful in studying modular fusion categories. Unlike the case of minimal modular extensions, which as mentioned in Remark~\ref{remark:mne} do not always exist, Schopieray conjectures that Galois-modular extensions always exist \cite{SchopierayBIRS}. The goal of this paper is to prove Schopieray's conjecture in the pseudounitary  case:
\begin{mainthm}
  Every braided fusion category $\cB$ admits a Galois nondegenerate extension $\cB \mono \cM$. If $\cB$ is pseudounitary, then $\cM$ can be chosen to be pseudounitary as well, and hence (Galois-)modular.
\end{mainthm}

Our proof is not functorial but it is constructive in the sense that we produce such an extension subject to some noncanonical choices. The proof in the case when $Z_2(\cB)$ is \define{Tannakian} --- equivalent to $\Rep(G)$ for some finite group $G$ --- is given in  Section~\ref{sec:proof}, and in Section~\ref{sec:super} we generalize our proof to the  non-Tannakian case. 

%\begin{remark}
%  A choice of Galois-modular extension $\cB \mono \cM$ selects in particular an action of $\Gal(\bar{\bQ}/\bQ)$ on the set of simple objects in $\cB$. It would be interesting to compute this action in some examples of our construction. We have not attempted this.
%\end{remark}

\section{The main construction}\label{sec:proof}

The goal of this section is to prove our Main Theorem for braided fusion categories with Tannakian centre. In other words, we are given a
braided category $\cB$ and a braided equivalence $Z_2(\cB) \cong \Rep(G)$, and we wish to produce a nondegenerate extension $\cB \mono \cM$ such that $Z_2(\cB \mono \cM)$ is integral; we furthermore wish to arrange this extension so that if $\cB$ is pseudounitary, then so is $\cM$.

For any symmetric fusion category $\cE$, write $\NBFC(\cE)$ for the set of isomorphism classes of pairs consisting of a braided fusion category $\cB$ and a braided equivalence $Z_2(\cB) \cong \cE$. Given $\cB, \cD \in \NBFC(\cE)$, let us write $\cB \sim \cD$ if there exists a nondegenerate extension $\cB \mono \cM$ with $Z_2(\cB \mono \cM) \cong \cD^{\mathrm{rev}}$ (the braiding-reversal of $\cD$). 
The relation $\sim$ is an equivalence relation. Indeed, recall from \cite{MR3022755} that $\NBFC(\cE)$ is a commutative monoid under $\boxtimes_{\cE}$ and has a quotient group $\Witt(\cE)$ in which one quotients by those categories of the form $Z_2(\cE \mono Z_1(\cF))$, where $\cF$ is fusion, $Z_1(\cF)$ is its Drinfel'd centre, and $\cE \mono Z_1(\cF)$ is a braided fusion inclusion. The construction $\cE \leadsto \Witt(\cE)$ is functorial (c.f.\ Remark~2.8 of \cite{MR3022755}) for symmetric monoidal functors between symmetric fusion categories, and so in particular the unit inclusion $\Vec \to \cE$ selects a canonical homomorphism $\Witt := \Witt(\Vec) \to \Witt(\cE)$.

\begin{proposition}\label{thm:coker}
$\cB \sim \cD$ if and only if they represent the same class in the cokernel of the canonical map $\Witt \to \Witt(\cE)$.
\end{proposition}
\begin{proof}
  By
  Theorem~3.2 of \cite{MR4560997} (which handles the ``if'' direction) and  Proposition~4.1 of \cite{MNE} (which handles the ``only if'' direction), a braided fusion category $\cB \in \NBFC(\cE)$ admits a minimal nondegenerate extension if and only if its Witt class $[\cB] \in \Witt(\cE)$ is in the image of the map $\Witt \to \Witt(\cE)$. Suppose that $\cB$ and $\cD$ represent the same class in $\operatorname{coker}(\Witt \to \Witt(\cE))$. Then $\cB \otimes_\cE \cD^{\mathrm{rev}}$ represents the trivial class and so admits a minimal nondegenerate extension $\cM$; the composite extension $\cB \mono \cB \otimes_\cE \cD^{\mathrm{rev}} \mono \cM$ shows that $\cB \sim \cD$. Conversely, if there is an extension $\cB \mono \cM$ with centralizer $\cD$, then $\cB \otimes_\cE \cD^{\mathrm{rev}} \mono \cM$ is minimal nondegenerate.
\end{proof}

When $\cE = \Rep(G)$, then a completely concrete description is available. Write $\NBFC := \NBFC(\Vec)$ and consider the map $\NBFC(\cE) \to \NBFC$ that sends $\cB$ to its deequivariantization $\cB_G$ as in \cite[Theorem 4.18]{MR2183279}. 
Let $\NBFC^\triv(\Rep(G))$ denote the subset of $\NBFC(\Vec)$ on those $\cB$ for which $\cB_G$ is Witt-trivial, and $\Witt^\triv(\Rep(G))$ its image in $\Witt(\cE)$. Then we find a canonical splitting $\Witt(\Rep(G)) = \Witt \oplus \Witt^\triv(\Rep(G))$. The following is the main theorem of \cite{2601.09060} (resolving Conjecture~4.3 of \cite{MNE}):

\begin{proposition}[\cite{2601.09060}]\label{thm:JOY}
   For any finite group $G$, there is a canonical isomorphism\\[4pt] \mbox{} \hfill $O_4 : \Witt^\triv(\Rep(G)) \to \H^4(\rB G; \bC^\times).$  \qed
\end{proposition}

For us, the most interesting part of this Proposition is the rather explicit construction, given in Proposition~5.8 of \cite{2601.09060}, of an element $\cB_\omega \in \NBFC^\triv(\Rep(G))$ for each $\omega \in \H^4(\rB G; \bC^\times)$. 
The specifics of the construction will almost immediately imply our Main Theorem, and so we review the construction here. 

Fix a cocycle representing $\omega$ (and abusively write $\omega$ also for this cocycle representative). By the main theorem of \cite{MR1296597}, there exists a finite group $\tilde{G}$ with a surjection $\tilde{G} \to G$ such that the restriction $\omega|_{\tilde{G}}$ is exact; choose a primitive $\psi$, i.e.\ a 3-cochain on $\tilde{G}$ with $\d\psi = \omega$. An explicit and elementary construction of $(\tilde{G}, \psi)$ can be found for example in~\cite{2601.04374}.

Now consider the semisimple category $\Vec[\tilde{G}]$ of $\tilde{G}$-graded vector spaces. We will equip it with its standard binary convolution product $\otimes : \Vec[\tilde{G}] \times \Vec[\tilde{G}] \to \Vec[\tilde{G}]$, but we will not use its standard associator. Instead, we decide to use the cochain $\psi$ as the associativity isomorphism $(X \otimes Y) \otimes Z \cong X \otimes (Y \otimes Z)$, and denote the resulting category ``$\Vec^\psi[\tilde{G}]$.'' This is not a monoidal category: the pentagon equation for the associator does not hold. Indeed, the pentagon computes $\d\psi = \omega \neq 1$.
Instead, we equip $\Vec^\psi[\tilde{G}]$ with its $G$-grading via the map $\tilde{G} \to G$; this $G$-grading is faithful because this map is a surjection. In this way, $\Vec^\psi[\tilde{G}]$ becomes \define{$G$-quasimonoidal with pentogonicity-defect $\omega$} in the sense of \S4 of \cite{2601.09060}: a faithfully-$G$-graded category with a binary multiplication and associator which satisfies a the obvious $\omega$-twisted version of the pentagon axiom. It is moreover rigid (the notion of ``rigid'' for monoidal categories does not use the pentagon axiom) and semisimple with simple unit, and we will abbreviate ``rigid semisimple $G$-quasimonoidal with simple unit'' to \define{$G$-quasifusion}.

Suppose that $\cF$ is a $G$-quasifusion category. Write $\cF_e \subset \cF$ for the neutrally-graded subcategory. The methods of \S8 of \cite{MR2677836} supply a categorical $G$-action on the Drinfel'd centre $Z_1(\cF_e)$ --- indeed, they show that categorical $G$-actions on $Z_1(\cF_e)$ are precisely the same as $G$-quasifusion extensions of $\cF_e$. Consider the equivariantization $Z_1(\cF_e)^G$. This is manifestly an element of $\NBFC^\triv(\Rep(G))$ (its equivariantization is $Z_1(\cF_e)$). The methods of \S8 of \cite{MR2677836} moreover identify $O_4(Z_1(\cF_e)^G) \in \H^4(\rB G; \bC^\times)$ with the pentogonicity defect of $\cF$.

In our case, the neutrally-graded component of $\cF = \Vec^\psi[\tilde{G}]$ is $\Vec^\psi[K]$, where $K:=\ker(\tilde{G} \to G)$; note that $\psi|_K$ is closed.

Now return to the original problem at hand. We start with $\cB \in \NBFC(\Rep(G))$. Compute its image $O_4(\cB) \in \H^4(\rB G; \bC^\times)$ under the composite $\NBFC(\Rep(G)) \to \Witt(\Rep(G)) = \Witt \oplus \H^4(\rB G; \bC^\times) \epi \H^4(\rB G; \bC^\times)$. Set $\omega := O_4(\cB)^{-1}$, and build $(\Vec^\psi[K])^G$ as above. Then by Propositions~\ref{thm:coker} and~\ref{thm:JOY}, there is a nondegenerate extension $\cB \mono \cM$ with $Z_2(\cB \mono \cM) \cong Z_1(\Vec^\psi[K])^G$. But $Z_1(\Vec^\psi[K])^G$ is manifestly integral. We have shown:

\begin{proposition}\label{prop:summarytannakian}
  Let $\cB$ be a braided fusion category with Tannakian centre. Then $\cB$ admits a Galois nondegenerate extension $\cB \mono \cM$. \qed
\end{proposition}

To complete the proof of our Main Theorem, it remains to make this extension $\cB \mono \cM$ modular, i.e.\ to supply a spherical structure on $\cM$ extending the spherical structure on $\cB$. This author is not aware of any general results about modular extensions analogous to Proposition~\ref{thm:coker} about nondegenerate extensions. Instead, we use pseudounitarity to sidestep the issue.
Recall the following notion from \cite{MR2183279}: a fusion category is \define{pseudounitary} when it admits a (necessarily unique) ``positive'' spherical structure for which quantum and Frobenius--Perron dimensions agree.
 Note that $Z_1(\Vec^\psi[K])^G$ is manifestly pseudounitary. The following is essentially Corollary~1.3 of \cite{MNE}:

\begin{proposition}\label{prop:pseudounitary}
  Suppose that $\cB \mono \cM$ is a nondegenerate extension such that $\cB$ and $Z_2(\cB \mono \cM)$ are both pseudounitary. Then $\cM$ is pseudounitary.
\end{proposition}
\begin{proof}
  For a  fusion category $\cB$, its \define{Frobenius--Perron dimension} $\FPdim(\cB)$  is the sum of the squares of the Frobenius--Perron dimensions of all of its simple objects. Semisimplicity of $\cB$ provides a second notion of ``dimension,'' or more precisely ``squared dimension,'' of a simple object: for each simple object $x \in \cB$, semisimplicity implies that the left and right duals of $x$ are noncanonically isomorphic; choose such an isomorphism, and use it to define the left and right dimensions of $x$; the \define{squared dimension} $\dim^2 (x)$ is the product of its left and right dimensions (the choice of isomorphism cancels). The \define{global dimension} $\dim(\cB)$ of $\cB$ is the sum of the squared dimensions of the simple objects in $\cB$. 
  
  The squared dimension is never more than the square of the Frobenius--Perron dimension, and together with some statements about when this bound is saturated one arrives at 
  Proposition~8.23 of \cite{MR2183279}: 
  a braided fusion category $\cB$ is pseudounitary if and only if its global dimension $\dim (\cB)$ and Frobenius--Perron dimension $\FPdim (\cB)$ agree. 
  
  But on the other hand Theorems~3.10 and~3.14 of \cite{MR2609644} show that if $\cB \mono \cM$ is a nondegenerate extension, then
  \begin{gather*}
    \dim (\cM) = \dim (\cB) \times \dim Z_2(\cB \mono \cM), \\
    \FPdim (\cM) = \FPdim (\cB) \times \FPdim Z_2(\cB \mono \cM).
  \end{gather*}
  The result is then immediate.
\end{proof}

This proves our Main Theorem in the Tannakian case.

\begin{remark}\label{remark:alternate}
  Our proof of Proposition~\ref{prop:summarytannakian} is slightly nonconstructive: we show that $\cB$ admits a nondegenerate extension $\cM$ with $Z_2(\cB \mono \cM) \cong Z_1(\Vec^\psi[K])^G$ by showing that these categories represent opposite classes in $\operatorname{coker}(\Witt \to \Witt(\Rep(G)) \cong \H^4(\rB G; \bC^\times)$; to actually choose such an $\cM$ requires identifying these classes coherently.
  Let us now sketch a more categorical construction.

  As above, let $\cB$ be braided fusion with $Z_2(\cB) \cong \Rep(G)$. Deequivariantize to $\cB_G \in \NBFC$. The general theory of (de)equivariantization supplies a categorical action of $G$ on $\cB_G$ and an equivalence $\cB \cong (\cB_G)^G$. As above, choose an extension $K \to \tilde{G} \epi G$ with a trivialization of $O_4(\cB)|_{\tilde{G}}$. Now consider equivariantizing $\cB_G$ with respect to the induced action of $\tilde{G}$ --- let us call the result $\tilde{\cB} := (\cB_G)^{\tilde{G}}$. Then $O_4(\tilde{\cB}) = O_4(\cB)|_{\tilde{G}}$ is trivialized, and this trivialization selects a minimal nondegenerate extension:
  $$  \tilde\cB \mono \cM, \quad Z_2(\tilde\cB \mono \cM) \cong \Rep(\tilde{G}). $$
  There is a canonical inclusion $\cB \cong (\cB_G)^G \mono (\cB_G)^{\tilde{G}} = \tilde{\cB}$ (any categorical $G$-fixed point restricts along $\tilde G \to G$ to a categorical $\tilde{G}$-fixed point), leading to a composite nondegenerate extension
  $$ \cB \mono \tilde\cB \mono \cM. $$
  
  To understand $Z_2(\cB \mono \cM)$, note first that its centre matches the centre $Z_2(\cB) \cong \Rep(G)$ of $\cB$. Thus we do know that $Z_2(\cB \mono \cM)$ is the $G$-equivariantization of its $G$-deequivariantization, and so the interesting thing is to understand the nondegenerate braided fusion category $Z_2(\cB \mono \cM)_G$.
  
  There is a braided (but not central) inclusion $\Rep(\tilde G) \cong Z_2(\tilde\cB \mono \cM) \mono Z_2(\cB \mono \cM)$, leading to a braided (but not central) inclusion
  $$ \Rep(K) \cong \Rep(\tilde G)_G \mono Z_2(\cB \mono \cM)_G.$$
  On the other hand, an easy dimension computation shows that $\dim(Z_2(\cB \mono \cM)_G) = |K|^2$. In other words, $\Rep(K) \mono Z_2(\cB \mono \cM)_G$ is a minimal nondegenerate extension.
  
  But, for any finite group $K$, the minimal nondegenerate extensions of $\Rep(K)$ are completely classified in \cite{MR1923177} (see \S4.3 of \cite{MR3613518} for a nice treatment): they are precisely the categories of the form $Z_1(\Vec^\psi[K])$ for $\psi \in \H^3(\rB K; \bC^\times)$.
  
  Thus $Z_2(\cB \mono \cM) \cong Z_1(\Vec^\psi[K])^G$ for some $\psi$ and some action of $G$. One could work more to identify these data with the data in our earlier construction. If the goal is only to prove our Main Theorem, we do not need to: $Z_1(\Vec^\psi[K])$ and hence $Z_1(\Vec^\psi[K])^G$ are integral, and so $\cB \mono \cM$ is Galois.
\end{remark}

\section{The non-Tannakian case} \label{sec:super}

The goal of this section is to prove our Main Theorem in the case when $Z_2(\cB)$ is non-Tannakian. 
This case is slightly more difficult because there is not a clean statement like Proposition~\ref{thm:JOY} in the non-Tannakian case --- \S4.4 of \cite{MNE} sketches what the statement should be, but details are lacking. That said, significant partial results are available in \cite{MR4971769}, and these will suffice for our purposes. Our construction in this section will be an analogue of the construction from Remark~\ref{remark:alternate} rather than from Proposition~\ref{prop:summarytannakian} (which surely agree but we did not prove that they do).

Suppose that $\cB$ is a braided fusion category with non-Tannakian centre. The celebrated ``super fibre functor'' theorem of Deligne's \cite{MR1944506} implies that $Z_2(\cB)$ contains a unique maximal Tannakian subcategory $\Rep(G) \subset Z_2(\cB)$, and the deequivariantization $Z_2(\cB)_G$ is braided equivalent to $\sVec$. In particular, $\cB_G$ is a \define{slightly degenerate} braided fusion category: $Z_2(\cB_G) \cong \sVec$.

Write $\MNE(\cB_G)$ for the $2$-groupoid of minimal nondegenerate extensions of $\cB_G$.
The main theorem of \cite{MNE} states that $\MNE(\cB_G)$ is nonempty, and the construction in that paper gives a description of its homotopy groups (at any basepoint):
$$ |\pi_0 \MNE(\cB_G)| = 16, \quad \pi_1 \MNE(\cB_G) \cong \bZ/2\bZ, \quad \pi_2 \MNE(\cB_G) \cong \bZ/2\bZ.$$
The higher homotopy groups vanish because $\MNE(\cB_G)$ is a $2$-groupoid.

\begin{remark}\label{remark:LKWmext}
The main theorem of \cite{MR3613518} applied in this case states that $\pi_0 \MNE(\cB_G)$ is naturally a torsor for $\pi_0 \MNE(\sVec) \cong \bZ/16\bZ$, and the same arguments given therein but interpreted homotopically in fact show that the $\infty$-groupoid $\MNE(\sVec)$ is naturally monoidal and has $\MNE(\cB_G)$ as a torsor. (The paper \cite{MR3613518} states their results in terms of unitary (pre)modular categories. But the work applies without change to the purely algebraic setting of (non)degenerate braided fusion categories.) In particular the Postnikov $k$-invariant connecting $\pi_1\MNE(\cB_G)$ and $\pi_2\MNE(\cB_G)$ is independent of the choice of basepoint. We remark in passing that this $k$-invariant is nontrivial; we will not use it in the sequel.
\end{remark}

The categorical action of $G$ on $\cB_G$ induces an action on $\MNE(\cB_G)$. A (homotopy) fixed point of this action, if one exists, unpacks to a choice of minimal nondegenerate extension $\cB_G \mono \cM$ together with an extension of the $G$-action on $\cB_G$ to all of $\cM$.
The main theorem of \cite{MR4971769} asserts:
\begin{proposition}[Theorems~1.1 and~4.9 of \cite{MR4971769}]
  Suppose that $\cM \in \MNE(\cB_G)^G$ is a $G$-fixed point such that the corresponding $O_4$-obstruction $O_4(\cM) \in \H^4(\rB G; \bC^\times)$ is trivial. Choose a trivialization and construct the corresponding minimal nondegenerate extension $\cM^G \mono \cN$. Then the composite nondegenerate extension 
  $$ \cB \cong (\cB_G)^G \mono \cM^G \mono \cN $$
  is minimal, and all minimal nondegenerate extensions of $\cB$ are of this type. \qed
\end{proposition}

Thus the first step to constructing or obstructing a minimal nondegenerate extension of $\cB$ is to analyze the action of $G$ on $\MNE(\cB_G)$. In particular, constructing a fixed point can be reduced to a sequence of obstructions: one asks for a fixed point in $\pi_0 \MNE(\cB_G)$, then asks to lift this to the 1-truncation $\pi_{\leq 1}\MNE(\cB_G)$, and finally asks to lift this to the full 2-groupoid $\MNE(\cB_G)$.

The paper \cite{MR4971769} leaves open the possibility that $G$ acts nontrivially already on $\pi_0 \MNE(\cB_G)$. We do not have to worry about this possibility:

\begin{proposition}\label{prop:trivialaction}
  The action of $G$ on $\pi_0 \MNE(\cB_G)$ is trivial.
\end{proposition}
The original version of the present article supplied a proof of Proposition~\ref{prop:trivialaction} only when $\cB$ is pseudounitary, because it relied on the central charge statements from \cite{MR3641612} that are known only in the (pseudo)unitary case. The following improved proof was suggested by one of the anonymous referees:

\begin{proof}
  It suffices to prove that the sixteen minimal nondegenerate extensions of $\cB_G$ are pairwise inequivalent as braided fusion categories. 
  We will show that they are pairwise not even Witt-equivalent. 
  The torsorial action of $\pi_0 \MNE(\sVec)\cong \bZ/16\bZ$ on $\pi_0 \MNE(\cB_G)$ from \cite{MR3613518} is additive on Witt classes: if $\cC \in \MNE(\sVec)$ and $\cI \in \MNE(\cB_G)$, then on Witt classes we have
  $$ [\cC \otimes_{\sVec} \cI] = [\cC] + [\cI] \in \Witt.$$
  But $\MNE(\sVec) \to \Witt$ is an injection.
\end{proof}

Thus each point $[\cM] \in \pi_0\MNE(\cB_G)$ is trivially fixed under the $G$-action. Pick $\cM$ and ask: is it fixed in the 1-truncation $\pi_{\leq 1}\MNE(\cB_G)$? By general obstruction theory, the obstruction is some specific class $$O_2(\cB,\cM) \in \H^2(\rB G; \pi_1(\MNE(\cB_G),\cM)) = \H^2(\rB G; \bZ/2\bZ).$$ In general, such obstructions live in twisted cohomology (twisted by an action of $G$ on $\pi_1$); conveniently, $\bZ/2\bZ$ has no automorphisms so we do not need to worry about this issue. 

Suppose that $O_2(\cB,\cM)$ is trivializable and choose a trivialization $\phi$. Then there is a second obstruction
$$O_3(\cB, \cM, \phi) \in \H^3(\rB G;\pi_2 \MNE(\cB_G)) = \H^3(\rB G; \bZ/2\bZ).$$
Since that exhausts the nontrivial homotopy groups of $\MNE(\rB G)$, trivializing both obstructions is the same as making $\cM$ into a $G$-fixed point. Finally, to build a minimal nondegenerate extension, one faces a final obstruction $O_4(\cB,\cM,\phi,\rho) \in \H^4(\rB G; \bC^\times)$ which depends on a choice of trivialization $\rho$ of $O_3(\cB,\cM,\phi)$. These obstructions are also identified in \cite{MR4971769}.

\begin{remark}
  As emphasized by the notation, the values of these obstructions can a priori depend on the choice of $[\cM] \in \pi_0 \MNE(\cB_G)$ and on the trivializing cochains $\phi,\rho$. 
  The dependence on the trivializing cochains is essentially the nontriviality of the appropriate Postnikov invariants --- compare Remark~\ref{remark:LKWmext}. A version of the dependence on $\cM$ is explored in  \cite{2507.06304}.
\end{remark}

Fix $\cM \in \MNE(\cB_G)$. By repeatedly applying \cite{MR1296597}, although $\cM$ may not admit a $G$-fixed structure with trivial $O_4$-obstruction, we can certainly find some surjection of finite groups $\tilde G \epi G$ such that $\cM$ admits a fixed-point structure for the induced action of $\tilde{G}$ on $\MNE(\cB_G)$ and such that the induced $O_4$-obstruction is trivial in $\H^4(\rB \tilde G; \bC^\times)$. Said another way, we can definitely find a surjection $\tilde G \epi G$ such that $\tilde \cB := (\cB_G)^{\tilde G}$ admits a minimal nondegenerate extension, which we henceforth denote $\tilde \cM$.

\begin{proposition} The composite nondegenerate extension
$$ \cB \cong (\cB_G)^G \mono (\cB_G)^{\tilde G} \cong \tilde\cB \mono \tilde \cM$$
is Galois. \end{proposition}
\begin{proof}
  It suffices to control the centralizer $Z_2(\cB \mono \tilde \cM)$. Its centre agrees with the centre of $\cB$: it is super-Tannakian with maximal Tannakian subcategory $\Rep(G)$. Thus we can deequivariantize:
  $Z_2(\cB \mono \tilde \cM) = (Z_2(\cB \mono \tilde \cM)_G)^G$ and $Z_2(\cB \mono \tilde \cM)_G$ is slightly degenerate. Moreover, there is a braided inclusion
  $$ \Rep(\tilde G) \mono Z_2(\tilde \cB \mono \tilde \cM) \mono Z_2(\cB \mono \tilde \cM).$$
  This dequivariantizes to a braided inclusion
  $$ \Rep(K) \cong \Rep(\tilde G)_G \mono Z_2(\cB \mono \tilde \cM)_G.$$
  
  Counting dimensions, we find
  $$ \dim(Z_2(\cB \mono \tilde \cM)_G) = 2|K|^2.$$
  In other words, the inclusion $\Rep(K) \mono Z_2(\cB \mono \tilde \cM)_G$ is what one could call a ``minimal slightly degenerate extension.'' We will analyze these in analogy to the classification of minimal nondegenerate extensions of $\Rep(K)$ from \cite{MR1923177,MR3613518}.
  
  Let $\cN$ be an arbitrary slightly degenerate braided fusion category of dimension $\dim(\cN) = 2|K|^2$ equipped with a braided inclusion $\Rep(K) \mono \cN$. We can deequivariantize with respect to this inclusion. Since the inclusion is braided but not central, this deequivariantization $\cN_K$ will be monoidal but not braided. Rather, $\cN_K$ will be automatically a $K$-crossed braided extension. In particular, it will be faithfully $K$-graded, and the neutral component is central. Since $\cN$ was slightly degenerate, so will be the neutral component of $\cN_K$. Since $\cN$ had dimension $2|K|^2$, this neutral component will have dimension $2$, i.e.\ it will be $\sVec$. All together, we find that $\cN_K$ is a $K$-graded central extension of $\sVec$.
  
  Each component of $\cN_K$ will therefore be an invertible $\sVec$-module category. There are two such categories: the regular module $\sVec$ and the module $\Vec = \Mod_{\sVec}(\operatorname{Cliff}(1))$ coming from the monoidal but not braided functor $\sVec \to \Vec$. Given a simple object $x \in \cN_K$, one can quickly decide which of these two types is its component. Indeed,
  let $f \in \sVec$ denote the nontrivial simple object. If $f \otimes x \not\cong x$, the component containing $x$ is of type $\sVec$. In this case, the object $x$ is called \define{ordinary}. If $f \otimes x \cong x$, the component containing $x$ is of type $\Vec = \Mod(\operatorname{Cliff}(1))$. In this case $x$ is called \define{Majorana}.
  The dimensions of simple objects are easy to compute: $\dim(x) = 1$ if $x$ is ordinary and $\dim(x) = \sqrt{2}$ if $x$ is Majorana.
  
  The action of $K$ on $\cN_K$ cannot mix ordinary and Majorana objects. It follows that in $\cN \cong (\cN_K)^K$, the ordinary objects have integral dimensions whereas the Majorana objects have dimensions in $\bZ\sqrt{2}$.
  
  But in a slightly degenerate category, there are no Majorana objects! This is explained for example in the proof of Corollary~3.47 of \cite{MNE}. The proof given there: Let $\eta(x)$ denote the scalar defined by the composite
  $$ \mathbf{1} \overset{\operatorname{coev}_x}\longto x \otimes x^* \overset{\operatorname{br}_{x,x^*}}\longto x^* \otimes x \overset{\operatorname{ev}_x}\longto \mathbf{1}.$$ 
  Since $f$ is central, $\eta(fx) = \eta(f)\eta(x)$, and $\eta(f) = -1$; but $\eta(x) \neq 0$ if $x$ is simple.
   Thus $\cN$ is integral.
\end{proof}

Together with Proposition~\ref{prop:pseudounitary}, this proves our Main Theorem in the non-Tannakian case.

%\bibliography{ReferencesWithLinks}{}
%\bibliographystyle{alpha}

\newcommand{\etalchar}[1]{$^{#1}$}

\end{document}